\documentclass[12pt]{amsart}
\usepackage{epsfig}
\usepackage{setspace}
\usepackage{latexsym, amsmath}
\usepackage{amsfonts}
\usepackage{amssymb}
\usepackage{amscd}
\usepackage{amsthm}
\usepackage{algorithmic}
\usepackage{algorithm}
\usepackage[all]{xy}

\input xy
\xyoption{all}

\theoremstyle{plain}
\newtheorem{theorem}{Theorem}[section]
\newtheorem{proposition}[theorem]{Proposition}

\theoremstyle{definition}

\newtheorem{remark}[theorem]{Remark}

\newcommand{\bQ}{{\mathbb Q}}
\newcommand{\bZ}{{\mathbb Z}}

\newcommand{\cH}{{\mathcal H}}

\newcommand{\bC}{{\mathbb C}}

\begin{document}

\title{On some special classes of complex elliptic curves}
\author{ Bogdan Canepa }
\address{"Ovidius" University \\
124 Mamaia Blvd., 900527 Constanta, Romania } \email[Bogdan
Canepa]{bogdan\_canepa@yahoo.com}
\author{Radu Gaba}
\address{Institute of Mathematics "Simion
Stoilow" of the Romanian Academy\\ P.O. BOX 1-764 RO-014700
Bucharest, ROMANIA} \email[Radu Gaba]{radu.gaba@imar.ro}

\date{October 30, 2011}
\begin{abstract}
In this paper we classify the complex elliptic curves $E$  for
which there exist cyclic subgroups $C\leq (E,+)$ of order $n$ such
that the elliptic curves $E$ and $E/C$ are isomorphic, where $n$
is a positive integer. Important examples are provided in the last
section. Moreover, we answer the following question: given a
complex elliptic curve E, when can one find a cyclic subgroup  $C$
of order $n$ of $(E,+)$ such that
$(E,C)\sim(\frac{E}{C},\frac{E[n]}{C})$, $E[n]$ being the
$n$-torsion subgroup of $E$, classifying in this way the fixed
points of the action of the Fricke involution on
  the open modular curves $Y_0(n)$.\\

  \emph{Mathematics subject classification: Primary: ${\rm 11G07, 11G15}$, Secondary: ${\rm 14D22}$}

  \emph{Key words:} \text {elliptic curve, Fricke involution,
  modular curve}
\end{abstract}
\maketitle

\section{Introduction}

 Let $E$ be an elliptic curve defined over the field of complex numbers $\bC$ and $C$ be a subgroup (not necessarily cyclic) of
 order $n<\infty$ of $(E,+)$. This means that $C$ is a subgroup of order $n$
 of $E[n]$ where by $E[n]$ one denotes the $n$-torsion subgroup of $E$  i.e. the set of points of order $n$ in $E$: $E[n]=\{P\in E :
 [n]P=O\}$. Let also $\pi:E\rightarrow E/C$ be the natural projection. Since $C$ acts
  effectively and properly discontinuous on $E$, the group $E/C$ has a structure of Riemann
  variety, compatible with the morphism $\pi$ and moreover $\pi$ is unramified of degree $n$: ${\rm
  deg}\pi = |\pi^{-1}(O)| = |C| = n$. It is known that $E/C$ is a complex elliptic curve and that
 if $C$ is cyclic, one has $E[n]/C\cong \bZ/n\bZ$.

  Denote by $\cH$ the upper half plane i.e. $\{z\in \bC | \hspace{1mm} {\rm Im}(z) >
  0\}$. One defines the open modular curves $Y_0(n)$ as the quotient space $\Gamma_0(n)/ \cH$ that is
   the set of orbits $\{\Gamma_0(n)\tau : \tau \in \cH\}$, where
  $\Gamma_0(n)$ is the "Nebentypus" congruence subgroup of level
  $n$ of $SL_2(\bZ)$, acting on $\cH$ from the left:

  \begin{center}
   $\Gamma_0(n)=\{\left(\begin{array}{cc}
a&b\\c&d\end{array}\right)\in
   \mathbb{SL}_2(\mathbb{Z})|\hspace{1 mm} c\equiv 0 ({\rm mod} n)\}$.
  \end{center}
Following the notations of \cite{DS}, an enhanced elliptic curve
for $\Gamma_0(n)$ is by definition an ordered pair $(E,C)$ where
$E$ is a complex elliptic curve and  $C$ a cyclic subgroup of $E$
of order $n$. Two pairs $(E,C)$ and $(E',C')$ are said to be
equivalent if some isomorphism $E\cong E'$ takes $C$ to $C'$. One
denotes the set of equivalence classes with
\begin{center}
$S_0(n):=\{$enhanced elliptic curves for $\Gamma_0(n)\}/\sim$.
\end{center}
Furthermore, an element of $S_0(n)$ is an equivalence class
$[E,C]$. $S_0(n)$ is a moduli space of isomorphism classes of
complex elliptic curves and $n$-torsion data.

Denote now by $\Lambda_{\tau}$ the lattice $\bZ+\bZ \tau$, $\tau
\in \cH$ and by $E_{\tau}$ the elliptic curve
$\bC/\Lambda_{\tau}$. Then one has the following bijection:
\begin{center}
$S_0(n)\cong Y_0(n)$ given by $[\bC/\Lambda_{\tau}, \langle 1/n +
\Lambda_{\tau} \rangle]\mapsto \Gamma_0(n) \tau$  (see \cite{DS},
Theorem 1.5.1 for details.)
\end{center}
  In this paper we study the complex elliptic curves $E$ for which there exist
cyclic subgroups $C\leq (E,+)$ of order $n$ such that the elliptic
curves $E$ and $E/C$ are isomorphic, where $n$ is a positive
integer. We observe that $E$ (and consequently $E/C$) are CM
curves, obviously have isomorphic endomorphism rings and hence
these points $(E,C)$ with $E\cong E/C$ are a special class of
Heegner points on the modular curve $X_0(n)=Y_0(n)\cup \{ {\rm
cusps}\}$. Very nice examples are provided in the last section for
the cases $n=2,3,5$ (and we remark that for $p$ prime there are
exactly $p+1$ complex elliptic curves $E$ (up to an isomorphism)
which admit at least one subgroup $C\leq (E,+)$ of order $p$ such
that $\frac{E}{C}\simeq E$). We recall that the Heegner points of
$Y_0(n)$ as defined by Birch in \cite{B1} are pairs $(E,E')$ of
$n$-isogenous curves with the same endomorphism ring. Heegner was
the one introducing them in \cite{He} while working on the class
number problem for imaginary quadratic fields and their importance
is extensively described in the survey \cite{B2}.

After elementarily studying the above mentioned class of Heegner
points, upon imposing certain conditions (see Theorem
\ref{theorem:2} and Proposition \ref{proposition:0}), given a
complex elliptic curve $E$, we answer the question of when there
exists $C\leq(E,+)$ cyclic subgroup of order $n$ of $E$ such that
 $(E,C)\sim(\frac{E}{C},\frac{E[n]}{C})$, studying in this way the
fixed points of the action of the Fricke involution
\begin{center}
  $w_n:=\left(\begin{array}{cc}
0&-1\\n&0\end{array}\right) \in GL_2(\bQ^{+})$
  \end{center} on the open modular curves $Y_0(n)$.

 The number of fixed points of the Fricke involution on $Y_0(n)$ was computed by
 Ogg (see \cite{O}, Proposition 3) and Kenku (see \cite{K},
 Theorem 2) and  this number, for $n>3$, is  $\nu(n)=h(-n)+h(-4n)$ if $n\equiv 3 ({\rm
 mod}4)$ and $\nu(n)=h(-4n)$ otherwise, where $h(-n)$ is the class number of primitive quadratic forms of discriminant $-n$ and $\nu(2)=\nu(3)=2$.

   It is easy to
  see that $\sigma$ normalizes the group $\Gamma_0(n)$ and hence
  gives an automorphism $\Gamma_0(n)z\mapsto \Gamma_0(n) \sigma
  (z)$ on $Y_0(n)$. It is also easy to check that this
  automorphism is an involution. The following proposition is a known result which we will use throughout the paper and we  skip
  its proof (the reader may consult \cite{DR}, section IV), 4.4 or \cite{Hu}, Theorem 2.4 and Remark 5.5 for details):

  \begin{proposition}
The action of  $w_n$ on the moduli space $S_0(n)$ is given by:
$[E,C]\mapsto [E/C,E[n]/C]$.

  \end{proposition}

\section{Main results}

With the background from the previous section we are ready to
classify the complex elliptic curves $E$ for which there exist
cyclic subgroups $C\leq (E,+)$ of order $n$ such that the elliptic
curves $E$ and $E/C$ are isomorphic:

\begin{theorem}\label{theorem:1}

   Let $E$ be a complex elliptic curve determined by the lattice $\langle 1,\tau\rangle$, $\tau\in \cH$ . Then:\\
  \indent i) $\exists C\leq (E,+)$ finite cyclic subgroup such that $\frac{E}{C}\simeq E\Leftrightarrow\exists u,v\in\mathbb{Q}$
  such that $\tau^{2}=u\tau+v$ with $\Delta=u^{2}+4v<0$ (i.e. $E$ admits complex multiplication);\\

  \indent ii) If $\tau$ satisfies the conditions of i) and $u=\frac{u_{1}}{u_{2}}, v=\frac{v_{1}}{v_{2}},u_{2}\neq 0,v_{2}\neq 0,
   u_{1}, u_{2}, v_{1}, v_{2}\in\mathbb{Z}, (u_{1},u_{2})=(v_{1},v_{2})=1, d_{2}=(u_{2},v_{2})$, then:\\

    $\exists C\leq (E,+)$ cyclic subgroup of order $n$ which satisfies $\frac{E}{C}\simeq E \Longleftrightarrow
     \exists (a,b')\in\mathbb{Z}^{2}$ with $(a,b')=1$ such that $n={\rm det}
     M$, where $M$ is the matrix
     \begin{equation*}
M=\left(
\begin{array}{ccc}
a & A  \\
b & B \\
\end{array} \right)
 \end{equation*}

and  $(a,A,b,B)=\Big(a,\frac{u_{2}v_{1}}{d_{2}}b',\frac{u_{2}v_{2}}{d_{2}}b', a+\frac{u_{1}v_{2}}{d_{2}}b'\Big)$;\\

 \indent iii) The subgroup $C$ from ii) is C=$\langle\frac{u_{11}+u_{21}\tau}{n}\rangle$, where $u_{11},u_{21}$ are obtained in the following way:\\

 Since ${\rm det} M = n$ and ${\rm gcd}(a,A,b,B)=1$ (one deduces easily this), the matrix $M$
 is arithmetically equivalent with the matrix:
\begin{equation*}
M\sim\left(
\begin{array}{ccc}
1 & 0  \\
0 & n \\
\end{array} \right),
\end{equation*}
hence
\begin{equation*}
 \exists U,V\in GL_{2}(\mathbb{Z})\quad \text{such that}\quad M=U\cdot\left(
\begin{array}{ccc}
1 & 0  \\
0 & n \\
\end{array} \right)\cdot V.
\end{equation*}
The elements $u_{11},u_{21}$ are the first column of the matrix
\begin{equation*}
U=\left(
\begin{array}{ccc}
u_{11} & u_{12}  \\
u_{21} & u_{22} \\
\end{array} \right).
 \end{equation*}\\
\end{theorem}

\begin{proof}

 We have that L=$\mathbb{Z}+\mathbb{Z}\tau$ and $C=\langle\bar{g}\rangle\leq E=\frac{\mathbb{C}}{L}$
  cyclic subgroup of order $n$ hence $C=\frac{L+\mathbb{Z}g}{L}$ and $n\bar{g}=\bar{0}$. It follows that $ng\in L$ and hence we can write
   $g\in\mathbb{C}$ as
 \begin{equation}
  g=\frac{\alpha_{1}+\alpha_{2}\tau}{n}, \alpha_{1},\alpha_{2}\in\mathbb{Z}.
 \end{equation}\\
  Since ${\rm ord}(\bar{g})=n$ it follows easily that ${\rm gcd}(\alpha_{1},\alpha_{2},n)=1$.\\
 We have that $\frac{E}{C}\simeq\frac{\frac{\mathbb{C}}{L}}{\frac{L+\mathbb{Z}g}{L}}\simeq\frac{\mathbb{C}}{L+\mathbb{Z}g}=\frac{\mathbb{C}}{L'}$,
  where we denoted by $L'=L+\mathbb{Z}g=\mathbb{Z}+\mathbb{Z}\tau+\mathbb{Z}g$.\\
 It is known that $\frac{\mathbb{C}}{L}\simeq\frac{\mathbb{C}}{L'}\Leftrightarrow\exists \lambda\in\mathbb{C}$ such that $\lambda L=L'$.
 Consequently $E\simeq\frac{E}{C}\Leftrightarrow\frac{\mathbb{C}}{L}\simeq\frac{\mathbb{C}}{L'}\Leftrightarrow\exists
\lambda\in\mathbb{C}$ such that
 \begin{equation}
 \lambda(\mathbb{Z}+\mathbb{Z}\tau)=\mathbb{Z}+\mathbb{Z}\tau+\mathbb{Z}g.
 \end{equation}
 The relation (2) can be studied in the following way:\\

 "$\subseteq$" Since $g$ is given by (1), from the inclusion "$\subseteq$" one shows
 that (note that this is not an equivalence):
 \begin{center}
  $\exists a,b,A,B\in \mathbb{Z}$ such that
 \begin{equation}
 \begin{cases}
 \lambda=\frac{a+b\tau}{n}\\
 \lambda\tau=\frac{A+B\tau}{n}\\
 \end{cases}
 \end{equation}
 \end{center}

 "$\supseteq$" The inclusion "$\supseteq$" is equivalent with
 \begin{center}
 \begin{equation}
 \exists \alpha,\beta,u,v,s,t\in\mathbb{Z}\quad \text{such that}
\begin{cases}
1=\alpha\lambda+\beta\lambda\tau\\
\tau=u\lambda+ v\lambda\tau\\
g=s\lambda+ t\lambda\tau\\
\end{cases}
\end{equation}
\end{center}

By replacing (3) and (1) in (4) we obtain
\begin{equation}
\begin{cases}
1=\alpha\frac{a+b\tau}{n}+\beta\frac{A+B\tau}{n}\\
\tau=u\frac{a+b\tau}{n}+ v\frac{A+B\tau}{n}\\
\frac{\alpha_{1}+\alpha_{2}\tau}{n}=s\frac{a+b\tau}{n}+t\frac{A+B\tau}{n}\\
\end{cases}
\Longleftrightarrow
\begin{cases}
\left(
\begin{array}{ccc}
a & A  \\
b & B \\
\end{array} \right)\cdot
\left(
\begin{array}{ccc}
\alpha & u  \\
\beta & v \\
\end{array} \right)= nI_{2}\\
\\
\left(
\begin{array}{ccc}
a & A  \\
b & B \\
\end{array} \right)\cdot
\left(
\begin{array}{ccc}
s \\
t \\
\end{array} \right)=\left(
\begin{array}{ccc}
\alpha_{1} \\
\alpha_{2} \\
\end{array} \right)
\end{cases}
\end{equation}\\
\end{proof}

Note that the relation (2) is not yet equivalent with (3) and (5).
In order to obtain the equivalent conditions for (2) let us look at the inclusion "$\subseteq$".\\
The inclusion "$\subseteq$" means
 that there exist $a_{1},b_{1},c_{1},a_{2},b_{2},c_{2}\in \bZ$ such that
\begin{center}
\begin{equation}
\begin{cases}
\lambda=a_{1}+b_{1}\tau+c_{1}g\\
\lambda \tau=a_{2}+b_{2}\tau+c_{2}g\\
\end{cases}
\Longleftrightarrow
\begin{cases}
\frac{a+b\tau}{n}=a_{1}+b_{1}\tau+c_{1}\frac{\alpha_{1}+\alpha_{2}\tau}{n}\\
\frac{A+B\tau}{n}=a_{2}+b_{2}\tau+c_{2}\frac{\alpha_{1}+\alpha_{2}\tau}{n}\\
\end{cases}
\Longleftrightarrow
\end{equation}
\begin{equation*}
\begin{cases}
a=na_{1}+ c_{1}\alpha_{1}\\
b=nb_{1}+ c_{1}\alpha_{2}\\
\end{cases}
\text{and}
\begin{cases}
A=na_{2}+ c_{2}\alpha_{1}\\
B=nb_{2}+ c_{2}\alpha_{2}\\
\end{cases}
\Leftrightarrow
\end{equation*}\\
\begin{equation}
\begin{cases}
\hat{a}=\hat{ c_{1}}\hat{\alpha_{1}}\\
\hat{b}=\hat{ c_{1}}\hat{\alpha_{2}}\\
\end{cases}
\text{and}
\begin{cases}
\hat{A}=\hat{ c_{2}}\hat{\alpha_{1}}\\
\hat{B}=\hat{ c_{2}}\hat{\alpha_{2}}\\
\end{cases}
\text{in}\quad\frac{\bZ}{n \bZ}
\end{equation}
\end{center}

We obtain that the inclusion $"\subseteq"$ is equivalent with the
existence of  $c_{1},c_{2}\in\mathbb{Z}$ which both satisfy (6),
i.e.  (2) is equivalent with (5) and (6). \\

We've obtained that the hypothesis $E\simeq\frac{E}{C}$ with $C$
cyclic subgroup of order $n$ is equivalent with the relations (2)
 and $(\alpha_{1},\alpha_{2},n)=1$, i.e. is equivalent with the relations (5), (6), $(\alpha_{1},\alpha_{2},n)=1$ and $(a+b\tau)\tau=A+B\tau$.
  Assume now that these four conditions are satisfied and let us replace them with
  conditions which are easier to be verified.\\

Denote by $M$ the matrix
      \begin{equation*}
M=\left(
\begin{array}{ccc}
a & A  \\
b & B \\
\end{array} \right)
\text{hence}\quad \hat{M}=\left(
\begin{array}{ccc}
\hat{a} &\hat{A}  \\
\hat{b} &\hat{B}  \\
\end{array} \right)=\left(
\begin{array}{ccc}
\hat{ c_{1}}\hat{\alpha_{1}} & \hat{ c_{2}}\hat{\alpha_{1}} \\
 \hat{ c_{1}}\hat{\alpha_{2}}& \hat{ c_{2}}\hat{\alpha_{2}}\\
\end{array} \right)
\text{in}\quad\frac{\bZ}{n \bZ} \Longrightarrow
\end{equation*}
\begin{equation*}
 {\rm det}\hat{M}=\hat{0}\quad \text{in}\quad\frac{\bZ}{n \bZ}
 \Longrightarrow \quad \text{det M} \vdots  n
 \end{equation*}\\
By using now (5) we obtain ${\rm det}(M)|n^{2}$ and so ${\rm
det}(M)=nk, k|n$. From (5) we get $\alpha=\frac{nB}{det
M}=\frac{B}{k}$
 hence $k|B$. Similarly one obtains $k|b,k|A,k|a$ and hence from the second equality of (5) we obtain that $k|\alpha_{1},\alpha_{2}$.
 Since $k|n$ and $(\alpha_{1},\alpha_{2},n)=1$ we obtain $k=\pm 1$, consequently ${\rm det}(M)=\pm n$.\\

We will prove that the relations (5), ${\rm gcd}(\alpha_{1},\alpha_{2},n)=1$ and ${\rm det}(M)=\pm n$ imply (6):\\

  \indent In did, from $(\alpha_{1},\alpha_{2},n)=1$  it follows that $\exists \mu_{1},\mu_{2},\mu_{3}\in\mathbb{Z}$
  such that $\mu_{1}\alpha_{1}+ \mu_{2}\alpha_{2}+ \mu_{3}n=1$.
  \indent Choose  $\hat{c}_{1}=\widehat{\mu_{1}a+\mu_{2}b}$. We
  then have
  $\hat{c}_{1}\hat{\alpha}_{1}=\widehat{(\mu_{1}a+\mu_{2}b)}\hat{\alpha}_{1}= a\cdot\widehat{\mu_{1}\alpha_{1}}+ b\cdot\widehat{\mu_{2}\alpha_{1}}=a\widehat{(1-\mu_{2}\alpha_{2}- \mu_{3}n)}+ b\cdot\widehat{\mu_{2}\alpha_{1}}=\hat{a}+\hat{\mu}_{2}\widehat{(b\alpha_{1}-a\alpha_{2})}$.\\
  From (5) we get that $t=\frac{a\alpha_{2}-b\alpha_{1}}{\text{det} M}=\frac{a\alpha_{2}-b\alpha_{1}}{\pm n}\Longrightarrow a\alpha_{2}-b\alpha_{1}\vdots n$.
  Consequently $\hat{c}_{1}\hat{\alpha}_{1}=\hat{a}$, q.e.d.\\
  Similarly $\hat{c}_{1}\hat{\alpha}_{2}=\widehat{(\mu_{1}a+\mu_{2}b)}\hat{\alpha}_{2}= a\cdot\widehat{\mu_{1}\alpha_{2}}+ b\cdot\widehat{\mu_{2}\alpha_{2}}=a\cdot\widehat{\mu_{1}\alpha_{2}}+ b\widehat{(1-\mu_{1}\alpha_{1}- \mu_{3}n)} =\hat{b}+\hat{\mu}_{1}\widehat{(a\alpha_{2}-b\alpha_{1})}=\hat{b}$, q.e.d.\\
  Similarly, by choosing $\hat{c}_{2}=\widehat{\mu_{1}A+\mu_{2}B}$ one obtains the expressions of (6) regarding $A,B$.\\
  We have obtained that (6) follows from the relations  ${\rm gcd}(\alpha_{1},\alpha_{2},n)=1$, (5) and ${\rm det} (M)=\pm n$.\\
  Moreover, we can renounce at the first equality of (5) since it can be deduced from the condition ${\rm det} (M)=\pm n$:\\
   \indent In did, if ${\rm det} (M)=\pm n$ we have $M\cdot M^{*}={\rm det} (M)\cdot I_{2}=(\pm n)I_{2}$,
   hence the existence of the elements
   $\alpha,\beta, u, v$ follows from the existence of the adjoint matrix $ M^{*}$.\\
  In conclusion, the hypothesis $E\simeq\frac{E}{C}$ with $C$ cyclic subgroup of order $n$ is equivalent with the following four relations:
  the second equality from (5),  $(\alpha_{1},\alpha_{2},n)=1$, ${\rm det} (M)=\pm n$  and $(a+b\tau)\tau=A+B\tau$.\\

   We prove now that we can renounce at the second equality of (5) and at ${\rm gcd}(\alpha_{1},\alpha_{2},n)=1$,
    by replacing them in the above equivalence with the condition  ${\rm gcd}(a,A,b,B)=1$ :\\

  Let $d={\rm gcd}(a,A,b,B)$. From the second equality from (5) we obtain  $d|\alpha_{1}$, $d|\alpha_{2}$ and since $d|\text{det} M=\pm n$,
  it follows that $d|{\rm gcd}(\alpha_{1},\alpha_{2},n)=1$, i.e. ${\rm gcd}(a,A,b,B)=1$. We have that $M$ is arithmetically equivalent with:\\
 \begin{equation*}
 M=\left(
\begin{array}{ccc}
a & A  \\
b & B \\
\end{array} \right)\sim
\left(
\begin{array}{ccc}
1 & 0 \\
0 & n \\
\end{array} \right),\quad \text{i.e.}\quad \exists U,V\in SL_{2}(\mathbb{Z})\quad \text{such that}
\end{equation*}
\begin{equation*}
\left(
\begin{array}{ccc}
a & A  \\
b & B \\
\end{array} \right)=U\cdot
\left(
\begin{array}{ccc}
1 & 0 \\
0 & n \\
\end{array} \right)\cdot V
 \end{equation*}

 The second condition of (5) becomes:
\begin{equation*}
 \left(
\begin{array}{ccc}
a & A  \\
b & B \\
\end{array} \right)\cdot
\left(
\begin{array}{ccc}
s \\
t \\
\end{array} \right)=\left(
\begin{array}{ccc}
\alpha_{1} \\
\alpha_{2} \\
\end{array} \right)\Longleftrightarrow
U\cdot \left(
\begin{array}{ccc}
1 & 0 \\
0 & n \\
\end{array} \right)\cdot V\cdot
\left(
\begin{array}{ccc}
s \\
t \\
\end{array} \right)=\left(
\begin{array}{ccc}
\alpha_{1} \\
\alpha_{2} \\
\end{array} \right).
\end{equation*}
\begin{equation*}
\text{We denote by}\quad V\cdot \left(
\begin{array}{ccc}
s \\
t \\
\end{array} \right)=
\left(
\begin{array}{ccc}
s' \\
t' \\
\end{array} \right)
\text{and obtain the equivalent equation}\quad
\end{equation*}
\begin{equation*}
\left(
\begin{array}{ccc}
\alpha_{1} \\
\alpha_{2} \\
\end{array} \right)=U\cdot
\left(
\begin{array}{ccc}
s'\\
nt'\\
\end{array} \right), U\in SL_{2}(\mathbb{Z}).
\end{equation*}

Viceversa, remark that the second condition of (5) follows from
${\rm det}(M)=\pm{n}$ and from the relation ${\rm gcd}(a,A,b,B)=1$
as it follows: choose $s',t'\in\mathbb{Z}$ such that $(s',n)=1$.
Consider the matrix $(\alpha_{1}\quad \alpha_{2})=(s'\quad
nt')\cdot U^{t}$, where  $U$ is an invertible matrix from the
arithmetically equivalent decomposition of $M$. We have that ${\rm
gcd}(\alpha_{1},\alpha_{2},n)={\rm
gcd}((\alpha_{1},\alpha_{2}),n)={\rm gcd}(( s',nt'),n)={\rm
gcd}(s',n)=1$
  and the second equality of (5) follows from the previous construction, q.e.d.\\

Let us analyze now the condition $(a+b\tau)\tau=A+B\tau$. This
means that $\tau$
 is algebraic over $\mathbb{Q}$ and satisfies the second degree equation $b\tau^{2}=(B-a)\tau+A$ $(*)$.\\
Let $\mu_{\tau}=X^{2}-uX-v\in \mathbb{Q}[X]$ be the minimal
polynomial of $\tau$, with $u=\frac{u_{1}}{u_{2}}\in \mathbb{Q},
v=\frac{v_{1}}{v_{2}}\in \mathbb{Q}, u_{1}, u_{2}, v_{1},
v_{2}\in\mathbb{Z}, {\rm gcd}(u_{1},u_{2})={\rm
gcd}(v_{1},v_{2})=1$ and let $d_{2}:={\rm gcd}(u_{2},v_{2})$. We
identify now the coefficients of the equation $(*)$ and of the
minimal polynomial of $\tau$ and obtain:
\begin{center}
\begin{equation*}
\begin{cases}
\frac{A}{b}=v=\frac{v_{1}}{v_{2}}\\
\frac{B-a}{b}=u=\frac{u_{1}}{u_{2}}\\
\end{cases}
\text{hence}
\begin{cases}
v_{2}A=v_{1}b\\
u_{2}(B-a)=u_{1}b\\
\end{cases}
\end{equation*}
\end{center}
\begin{center}
\begin{equation*}
\text{Since}
\begin{cases}
v_{2}A=v_{1}b\\
{\rm gcd}(v_{1},v_{2})=1\\
 \end{cases}
 \text{we get that}
 \begin{cases}
v_{2}|b\\
v_{1}|A\\
\end{cases}
\end{equation*}
\end{center}
\begin{center}
\begin{equation*}
\text{From}
\begin{cases}
u_{2}(B-a)=u_{1}b\\
{\rm gcd}(u_{1},u_{2})=1\\
 \end{cases}
 \text{it follows that}
 \begin{cases}
u_{2}|b\\
u_{1}|B-a\\
\text{we already have}\quad v_{2}|b\\
\end{cases}
{\rm hence} \ \
 {\rm lcm}[u_{2},v_{2}]|b.
\end{equation*}
\end{center}
\begin{center}
\begin{equation*}
\text{Consequently}\quad \exists b'\in \mathbb{Z}\quad \text{such
that}
\begin{cases}
b=[u_{2},v_{2}]b'=\frac{u_{2}v_{2}}{d_{2}}b'\\
v_{2}A=v_{1}b\\
u_{2}(B-a)=u_{1}b\\
\end{cases}
\text{hence}
\begin{cases}
b=\frac{u_{2}v_{2}}{d_{2}}b'\\
A=\frac{u_{2}v_{1}}{d_{2}}b'\\
B=a+\frac{u_{1}v_{2}}{d_{2}}b'\\
\end{cases}
\end{equation*}
\end{center}
where $d_{2}={\rm gcd}(u_{2},v_{2})$ and $\tau^{2}=u\tau+v$.\\

We prove now the equivalence: \\
\begin{equation*}
{\rm gcd}(a,A,b,B)=1 \Longleftrightarrow {\rm gcd}(a,b')=1.
\end{equation*}
$"\Longrightarrow"$ Denote by $d'={\rm gcd}(a,b')$. It follows
that
 $d'|{\rm gcd}(a,A,b,B)$ hence $d'|1$, i.e. $d'=1$.\\
$"\Longleftarrow"$ Let $d={\rm gcd}(a,A,b,B)$. Then $d|a$ and
since ${\rm gcd}(a,b')=1$ we get ${\rm gcd}(d,b')=1$. Moreover
$d_{2}=(u_{2},v_{2})$ so let $u_{2}=d_{2}u^{'}_{2}$ and
$v_{2}=d_{2}v^{'}_{2}$ with ${\rm gcd}(u^{'}_{2},v^{'}_{2})=1$.
Since $d|a$ and $d|B$ we have that $d|B-a$ and obtain:
\begin{center}
 \begin{equation*}
\begin{cases}
d|b=\frac{u_{2}v_{2}}{d_{2}}b'\\
d|A=\frac{u_{2}v_{1}}{d_{2}}b'=u^{'}_{2}v_{1}b'\\
d|B-a=\frac{u_{1}v_{2}}{d_{2}}b'=v^{'}_{2}u_{1}b'\\
\end{cases}{\rm hence}
\begin{cases}
d|\frac{u_{2}v_{2}}{d_{2}}=u^{'}_{2}v_{2}=u_{2}v^{'}_{2}\\
d|u^{'}_{2}v_{1}\\
d|v^{'}_{2}u_{1}\\
\end{cases}
\end{equation*}
\end{center}
 If $d|u^{'}_{2}$, since ${\rm gcd}(u^{'}_{2},v^{'}_{2})=1$ one obtains $d|u_{1}$. But ${\rm gcd}(u_{1},u_{2})=1$ and consequently $d=1$.\\
If $d|v_{2}$, using the fact that ${\rm gcd}(v_{1},v_{2})=1$ we
get that $d|u^{'}_{2}$. Since $d|u^{'}_{2}$ and ${\rm
gcd}(u_{1},u_{2})=1$ we have ${\rm gcd}(d,u_{1})=1$. Consequently
$d|v^{'}_{2}$ and since we already have $d|u^{'}_{2}$, we obtain
$d=1$.

 We prove now that $n={\rm det}(M)>0$:\\
 In order to simplify the notations, denote by $d:=d_{2}$ and note that $(a,A,b,B)=\Big(a,\frac{u_{2}v_{1}}{d}b',\frac{u_{2}v_{2}}{d}b', a+\frac{u_{1}v_{2}}{d}b'\Big)$.\\
 We obtain that
 ${\rm det}(M)=aB-bA=a\Big(a+\frac{u_{1}v_{2}}{d}b'\Big)-\Big(\frac{u_{2}v_{2}}{d}b'\Big)\Big(\frac{u_{2}v_{1}}{d}b'\Big)
=a^{2}+\frac{u_{1}v_{2}}{d}ab'-\frac{v_{1}v_{2}u_{2}^{2}}{d^{2}}b'^{2}=\Big(a+\frac{u_{1}v_{2}}{2d}b'\Big)^{2}-\frac{u_{1}^{2}v_{2}^{2}}{4d^{2}}b'^{2}-\frac{v_{1}v_{2}u_{2}^{2}}{d^{2}}b'^{2}=\Big(a+\frac{u_{1}v_{2}}{2d}b'\Big)^{2}-\frac{u_{1}^{2}v_{2}^{2}+4v_{1}v_{2}u_{2}^{2}}{4d^{2}}b'^{2}=\Big(a+\frac{u_{1}v_{2}}{2d}b'\Big)^{2}-\frac{u_{2}^{2}v_{2}^{2}\Big[\Big(\frac{u_{1}}{u_{2}}\Big)^{2}+4\frac{v_{1}}{v_{2}}\Big]}{4d^{2}}b'^{2}=\Big(a+\frac{u_{1}v_{2}}{2d}b'\Big)^{2}-\frac{u_{2}^{2}v_{2}^{2}\Delta}{4d^{2}}b'^{2}$.
  Since $\Delta<0$ and $u_{2}\neq 0, v_{2}\neq 0$ it follows easily that ${\rm  det}(M)>0$.\\

 In order to describe the subgroup $C$ we deduce that:
 \begin{center}
\begin{equation*}
 \left(
\begin{array}{ccc}
\alpha_{1} \\
\alpha_{2} \\
\end{array} \right)=U\cdot
\left(
\begin{array}{ccc}
s' \\
nt' \\
\end{array} \right)\Longleftrightarrow
\left(
\begin{array}{ccc}
\alpha_{1} \\
\alpha_{2} \\
\end{array} \right)=
\left(
\begin{array}{ccc}
u_{11} & u_{12}  \\
u_{21} & u_{22} \\
\end{array} \right)\cdot
\left(
\begin{array}{ccc}
s' \\
nt' \\
\end{array} \right)
\end{equation*}
\end{center}
\begin{center}
\begin{equation*}
\text{i.e.}\quad
\begin{cases}
\alpha_{1}=u_{11}s'+u_{12}nt'\\
\alpha_{2}=u_{21}s'+u_{22}nt'\\
\end{cases}.
\end{equation*}
\end{center}
We obtain that:
\begin{center}
\begin{equation*}
  g=\frac{\alpha_{1}+\alpha_{2}\tau}{n}=\frac{u_{11}s'+u_{12}nt'+(u_{21}s'+u_{22}nt')\tau}{n}=
  \end{equation*}
  \begin{equation*}
  \frac{u_{11}s'+u_{21}s'\tau+ n(u_{12}t'+ u_{22}t'\tau)}{n}=\frac{s'(u_{11}+u_{21}\tau)}{n}+ t'(u_{12}+ u_{22}\tau)=g'+t'(u_{12}+ u_{22}\tau),
 \end{equation*}
 \begin{equation*}
 \text{where we denoted by}\quad
 g'=\frac{s'(u_{11}+u_{21}\tau)}{n}.
 \end{equation*}
 \end{center}
 Consequently,
 $C=\frac{L+\mathbb{Z}g}{L}=\frac{L+\mathbb{Z}g'}{L}=\langle\overline{g'}\rangle$.
 But ${\rm gcd}(s',n)={\rm gcd}(\alpha_1,\alpha_2,n)=1$ hence there exist $a,b\in\mathbb{Z}$ such that $as'+bn=1$. It follows that $ag'=a\frac{s'(u_{11}+u_{21}\tau)}{n}=
 \frac{(1-bn)(u_{11}+u_{21}\tau)}{n}=\frac{u_{11}+u_{21}\tau}{n}-b(u_{11}+u_{21}\tau)$.
 We have obtained that $C=\frac{L+\mathbb{Z}g'}{L}=\langle\overline{g'}\rangle=\langle\frac{u_{11}+u_{21}\tau}{n}\rangle$.

 \begin{remark}
 For  $a,b'$ fixed integers such that ${\rm gcd}(a,b')=1$ the subgroup $C$ is uniquely determined hence we can denote  $C=C_{a,b'}$.
 \end{remark}

Recall now that for $E, E'$ complex elliptic curves and $C\leq
(E,+)$, $C'\leq(E',+)$ cyclic subgroups of order $n$ of $E$ and
$E'$ respectively, one has that
 $(E,C)\sim(E',C')\Longleftrightarrow\exists u:E\longrightarrow E'$ isomorphism such that $u(C)=C'$.\\
The question we want to answer is the following:\\

" Let $E$ be a complex elliptic curve. When $\exists C\leq(E,+)$
cyclic subgroup of order $n$ of $E$ such that
 $(E,C)\sim(\frac{E}{C},\frac{E[n]}{C})$, where $E[n]=\{x\in E:nx=0\}$ $?$"\\

In order to answer this question we will give the following
criterion:
\begin{theorem}\label{theorem:2} Let $E$ be an elliptic curve defined over $\mathbb{C}$ satisfying the conditions of \ref{theorem:1} i). Then the following are equivalent:
\begin{center}
 \indent $\{\exists C\leq(E,+)$ cyclic subgroup of order $n$ of $E$ such that
 $(E,C)\sim(\frac{E}{C},\frac{E[n]}{C})\}\Longleftrightarrow$ \\
 $\{
 \exists (a,b')\in\mathbb{Z}^{2}$ with ${\rm gcd}(a,b')=1$ such that ${\rm det}(M)=n$ and $n|{\rm Tr} (M)\}  \Longleftrightarrow$\\
 $\{\exists (a,b')\in\mathbb{Z}^{2}$,
 with ${\rm gcd}(a,b')=1$ such that ${\rm det}(M)=n$ and $M^{2}\equiv O_{2}({\rm mod}
 n)$\},
 \end{center}

where $M$ is the matrix from \ref{theorem:1} ii) and ${\rm Tr}
(M)$ the trace of $M$.
\end{theorem}

\begin{proof}
We use the notations of \ref{theorem:1}.\\
 Let $E=\frac{\mathbb{C}}{L}$ with $L=\mathbb{Z}+\mathbb{Z}\tau\subset\mathbb{C}$ lattice and $C=\langle \bar{g} \rangle \leq E=\frac{\mathbb{C}}{L}$
 subgroup cyclic of order $n$, consequently $C=\frac{L+\mathbb{Z}g}{L}$.
 We have that $\frac{E}{C}\simeq\frac{\frac{\mathbb{C}}{L}}{\frac{L+\mathbb{Z}g}{L}}\simeq\frac{\mathbb{C}}{L+\mathbb{Z}g}=\frac{\mathbb{C}}{L'}$, where we denoted by $L'=L+\mathbb{Z}g=\mathbb{Z}+\mathbb{Z}\tau+\mathbb{Z}g$.
Recall that $\frac{\mathbb{C}}{L}\simeq\frac{\mathbb{C}}{L'}\Leftrightarrow\exists \lambda\in\mathbb{C}$ such that $\lambda L=L'$.\\

Let $u$ be defined in the following way:
$E=\frac{\mathbb{C}}{L}\longrightarrow\frac{E}{C}=\frac{\mathbb{C}}{L'}=\frac{\mathbb{C}}{L+\mathbb{Z}g}$
isomorphism, $u(\bar{z})=\widehat{\lambda z}$.
 Then $u(C)=u\big(\frac{L+\mathbb{Z}g}{L}\big)=\lambda\cdot\frac{L+\mathbb{Z}g}{L'}=\frac{\lambda L+\mathbb{Z}\lambda g}{L'}$.
Since $\lambda L=L'$ we have that $u(C)=\frac{L'+
\mathbb{Z}\lambda g}{L'}=\langle \widehat{\lambda g},+ \rangle$.
But $E[n]=\frac{\frac{1}{n} L}{L}$ and consequently
$\frac{E[n]}{C}=\frac{\frac{\frac{1}{n}
L}{L}}{C}=\frac{\frac{\frac{1}{n}
L}{L}}{\frac{L+\mathbb{Z}g}{L}}=\frac{\frac{1}{n}
L}{L+\mathbb{Z}g}=\frac{\frac{1}{n} L}{L'}$. We obtain that:
\begin{center}
\begin{equation}
 u(C)=\frac{E[n]}{C}\Longleftrightarrow L'+ \mathbb{Z}\lambda g= \frac{1}{n}\cdot L
\end{equation}
\end{center}

The relation (7) is equivalent to $\lambda L+ \mathbb{Z}\lambda g=
\frac{1}{n}\cdot L\Longleftrightarrow
\lambda(L+\mathbb{Z}g)=\frac{1}{n}\cdot L\Longleftrightarrow
\lambda L'=\frac{1}{n}\cdot L\Longleftrightarrow \lambda^{2}
L=\frac{1}{n}\cdot L\Longleftrightarrow n \lambda^{2} L= L.$
 Moreover,
 \begin{center}
 \begin{equation}
  L=n\lambda^{2} L\Longleftrightarrow\exists M_{1}\in SL_{2}(\mathbb{Z})\quad \text{such that}
 \left(
\begin{array}{ccc}
n\lambda^{2} \\
n\lambda^{2}\tau \\
\end{array}\right)=M_{1}\cdot
\left(
\begin{array}{ccc}
1 \\
\tau \\
\end{array} \right)
\end{equation}
\end{center}

\begin{center}
\begin{equation*}
\text{We denote by}\quad M_{1}=\left(
\begin{array}{ccc}
a_{1} & A_{1}  \\
b_{1}& B_{1} \\
\end{array} \right)
 \text{and by using the relations (3):}
 \begin{cases}
 \lambda=\frac{a+b\tau}{n}\\
 \lambda\tau=\frac{A+B\tau}{n}\\
 \end{cases}\text{ we get:}
 \end{equation*}
 \begin{equation*}
 \left(
\begin{array}{ccc}
n\lambda^{2} \\
n\lambda^{2}\tau \\
\end{array} \right)=
\left(
\begin{array}{ccc}
a_{1} & A_{1}  \\
b_{1}& B_{1} \\
\end{array} \right)
\cdot \left(
\begin{array}{ccc}
1 \\
\tau \\
\end{array} \right)\Longleftrightarrow
\end{equation*}
\begin{equation*}
\begin{cases}
n\lambda^{2}=a_{1}+A_{1}\tau \\
n\lambda^{2}\tau=b_{1}+B_{1}\tau \\
\end{cases}\Longleftrightarrow
\begin{cases}
a_{1}+A_{1}\tau=\frac{(a+b\tau)^{2}}{n} \\
b_{1}+B_{1}\tau=\lambda\cdot(A+B\tau) \\
\end{cases}\Longleftrightarrow
\begin{cases}
a_{1}+A_{1}\tau=\frac{(a+b\tau)^{2}}{n} \\
b_{1}+B_{1}\tau=\frac{(a+b\tau)(A+B\tau)}{n} \\
\end{cases}
\end{equation*}
\end{center}

Note that $A+B\tau=\tau\cdot(a+b\tau)$, which implies that
$b\tau^{2}=(B-a)\tau+A$. By using this equation we obtain:
\begin{center}
$a_{1}+A_{1}\tau=\frac{(a+b\tau)^{2}}{n}\Longleftrightarrow
na_{1}+nA_{1}\tau=a^{2}+b^{2}\tau^{2}+2ab\tau=
a^{2}+2ab\tau+b[(B-a)\tau+A]=(a^{2}+bA)+b(a+B)\tau\Longleftrightarrow$
\begin{equation*}
\begin{cases}
a_{1}=\frac{a^{2}+b A}{n}\\
A_{1}=\frac{b(a+B)}{n}\\
\end{cases}
\end{equation*}
\end{center}

\begin{center}
$b_{1}+B_{1}\tau=\frac{(a+b\tau)(A+B\tau)}{n}=\frac{\tau(a+b\tau)^{2}}{n}\Longleftrightarrow
nb_{1}+nB_{1}\tau=\tau[(a^{2}+b A)+b(a+B)\tau]=(a^{2}+b
A)\tau+(a+B)((B-a)\tau+A)=(a+B)A+(B^{2}+b
A)\tau\Longleftrightarrow$
\begin{equation*}
\begin{cases}
b_{1}=\frac{(a+B)A}{n}\\
B_{1}=\frac{B^{2}+b A}{n}\\
\end{cases}
\end{equation*}
\end{center}
We have that:
\begin{center}
 \begin{equation*}
M^{2}=\left(
\begin{array}{ccc}
a & A  \\
b & B \\
\end{array} \right)\cdot
\left(
\begin{array}{ccc}
a & A  \\
b & B \\
\end{array} \right)=\left(
\begin{array}{ccc}
a^{2}+bA & A(a+B)  \\
b(a+B) & B^{2}+bA \\
\end{array} \right)=\left(
\begin{array}{ccc}
na_{1} & nb_{1}  \\
nA_{1} & nB_{1}\\
\end{array} \right).
\end{equation*}
\end{center}

It is easy to see that the relation (8) is equivalent to $M^{2}\equiv O_{2}({\rm mod} n)$:\\
$"\Longrightarrow"$ We proved above that $M^{2}=n\cdot M_{1}^{t}\equiv O_{2}({\rm mod} n)$.\\
$"\Longleftarrow"$ Since $M^{2}\equiv O_{2}({\rm mod} n)$ we
obtain that there exist $a_{1},b_{1},A_{1},B_{1}\in\mathbb{Z}$
such that
\begin{center}
 \begin{equation*}
M^{2}=\left(
\begin{array}{ccc}
na_{1} & nb_{1}  \\
nA_{1} & nB_{1}\\
\end{array} \right)=n\cdot M_{1}^{t},\quad \text{where we denote by}\quad M_{1}=
\left(
\begin{array}{ccc}
a_{1} & A_{1}  \\
b_{1}& B_{1} \\
\end{array} \right).
\end{equation*}
\end{center}
Since $M^{2}=n\cdot M_{1}^{t}$ and  ${\rm det} (M)=n$ we obtain
that $n^{2}=n^{2}\cdot {\rm det} (M_{1})$, hence ${\rm det}
(M_{1})=1$, i.e.
$M_{1}\in SL_{2}(\mathbb{Z})$. From the previous proof we obtain that the relation (8) is satisfied, q.e.d.\\

Moreover, from Hamilton-Cayley's theorem we know that
\begin{equation*}
M^{2}-(Tr (M))\cdot M+\text{det}(M)I_{2}=0_{2}.
\end{equation*}
Since ${\rm det} (M)=n$, $M\in SL_{2}(\mathbb{Z})$ and ${\rm
gcd}(a,A,b,B)=1$ it follows easily that
\begin{equation*}
M^{2}\equiv O_{2}\text{(mod} n)\Longleftrightarrow n|\text{Tr}(M)
\end{equation*}
which completes the proof.\\

\end{proof}

\begin{proposition}\label{proposition:0}

   Let $E$ be a complex elliptic curve satisfying the condition  of Theorem \ref{theorem:1} i),
  $\tau^{2}=u\tau+v,u,v\in\mathbb{Q}, \Delta=u^{2}+4v<0$
  (i.e. $E$ admits complex multiplication).\\
  \indent $\exists C\leq(E,+)$ cyclic subgroup of order $n$ of $E$ such that $(E,C)\sim(\frac{E}{C},\frac{E[n]}{C})$
  only when $n\in\{1,2,3,\frac{-u_{2}^{2}v_{2}^{2}\Delta}{4d^{2}},\frac{-u_{2}^{2}v_{2}^{2}\Delta}{d^{2}}\}$
  and moreover, this is happening only when the following conditions are
  satisfied: \\
  a) The case $n=\frac{-u_{2}^{2}v_{2}^{2}\Delta}{4d^{2}}$ occurs $\Longleftrightarrow 2d|u_{1}v_{2}$.\\
 This case is realized for $b'=\pm 1$ and  $a=-\frac{u_{1}v_{2}}{2d}b'$.\\
  b) The case $n=\frac{-u_{2}^{2}v_{2}^{2}\Delta}{d^{2}}$ occurs $\Longleftrightarrow 2d \not\vert u_{1}v_{2}$.\\
  This case is realized for $b'=\pm 2$ and $a=-\frac{u_{1}v_{2}}{2d}b'$.\\
  c) The case $n=2$ occurs $\Longleftrightarrow\frac{-u_{2}^{2}v_{2}^{2}\Delta}{4d^{2}}=1$ and $2d|u_{1}v_{2}$.\\
  This case is realized for $b'=\pm 1$ and $a=\pm 1 -\frac{u_{1}v_{2}}{2d}b'$.\\
  d) The case $n=3$ occurs $\Longleftrightarrow\frac{-u_{2}^{2}v_{2}^{2}\Delta}{d^{2}}=3$ and $2d \not\vert u_{1}v_{2}$.\\
  This case is realized for $b'=\pm 1$ and $a=\pm\frac{3}{2}-\frac{u_{1}v_{2}}{2d}b'$.\\
\end{proposition}

\begin{remark}
  For each complex elliptic curve satisfying the condition  of Theorem \ref{theorem:1} i) there exists a unique  $n\geq 2$
  for which $\exists C\leq(E,+)$ cyclic subgroup of order $n$ such
  that $(E,C)\sim(\frac{E}{C},\frac{E[n]}{C})$. In other words, if
  for a given $n$ there exist cyclic subgroups of order $n$ of $E$ such
  that $(E,C)\sim(\frac{E}{C},\frac{E[n]}{C})$ then for a
  different integer $m$ there are no cyclic subgroups of the same $E$
  such that $(E,C)\sim(\frac{E}{C},\frac{E[m]}{C})$.
  \end{remark}

\begin{proof}
From the proof of Theorem \ref{theorem:1} we have $(a,A,b,B)=\Big(a,\frac{u_{2}v_{1}}{d}b',\frac{u_{2}v_{2}}{d}b', a+\frac{u_{1}v_{2}}{d}b'\Big)$\\
and $n=\Big(a+\frac{u_{1}v_{2}}{2d}b'\Big)^{2}-\frac{u_{2}^{2}v_{2}^{2}\Delta}{4d^{2}}b'^{2}$.\\
 From Theorem \ref{theorem:2} we have  that there exist  $C\leq$ (E,+) cyclic subgroup of order $n$ of $E$ such that
  $(E,C)\sim(\frac{E}{C},\frac{E[n]}{C})$ if and only if $n={\rm det} M$ and $n|{\rm Tr}(M)=a+B$.\\
 The relation $n|{\rm Tr}(M)=a+B$ is equivalent to
\begin{equation}
 n=\Big(a+\frac{u_{1}v_{2}}{2d}b'\Big)^{2}-\frac{u_{2}^{2}v_{2}^{2}\Delta}{4d^{2}}b'^{2}\quad\vert\quad2\Big(a+\frac{u_{1}v_{2}}{2d}b'\Big)
\end{equation}\\
 By denoting $x=a+\frac{u_{1}v_{2}}{2d}b'$, (9) becomes
 \begin{equation}
 x^{2}-\frac{u_{2}^{2}v_{2}^{2}\Delta}{4d^{2}}b'^{2}\quad|\quad2x
 \end{equation}
 If $x>2$ then $x^{2}>2x$ and consequently the equation (10) is impossible.\\
 If $x<-2$ then $x^{2}>-2x$ and consequently the equation (10) is impossible.\\
We obtain that $x\in[-2,2]$, $x=\frac{m}{2}, m\in\mathbb{Z}$.\\
 \indent If $x=2$, by using the fact that $\Delta<0$, (10) becomes $4-\frac{u_{2}^{2}v_{2}^{2}\Delta}{4d^{2}}b'^{2}\quad|\quad4$, i.e. $b'=0$.
  Consequently $x=a=2$, $n=4$ and $(a, b')=(2,0)$. But ${\rm gcd} (a,b')=1$, contradiction.\\
  If $x=-2$ we obtain similarly that $b'=0$ and $x=a=-2$, hence ${\rm gcd} (a,b')=2$, contradiction.\\
If $x=0$ we have that $a=-\frac{u_{1}v_{2}}{2d}b'$. We distinguish two cases:\\
  \indent I If $2d|u_{1}v_{2}$ it follows that $b'|a$. Since ${\rm gcd} (a,b')=1$ it follows that $b'=\pm1$.\\
 Since $b'=\pm 1$ we have that $n=x^{2}-\frac{u_{2}^{2}v_{2}^{2}\Delta}{4d^{2}}b'^{2}=-\frac{u_{2}^{2}v_{2}^{2}\Delta}{4d^{2}}$.\\
    This case is realized for $b'=\pm 1$ and $a=-\frac{u_{1}v_{2}}{2d}b'$.\\
  \indent II If $2d\nmid u_{1}v_{2}$, by using the fact that $a=-\frac{u_{1}v_{2}}{2d}b'\in\mathbb{Z}$, we have
   $2|b'\Longrightarrow \exists b"\in\mathbb{Z}$ such that $b'=2b"$.
   Consequently $a=-\frac{u_{1}v_{2}}{d}b"$ and since ${\rm gcd}(a,b")=1$ we get that $b"=\pm1$ hence $b'=\pm 2$ .\\
   From $b'=\pm 2$ we obtain that $n=x^{2}-\frac{u_{2}^{2}v_{2}^{2}\Delta}{4d^{2}}b'^{2}=-\frac{u_{2}^{2}v_{2}^{2}\Delta}{d^{2}}$.\\
   This case is realized for $b'=\pm 2$ and $a=-\frac{u_{1}v_{2}}{2d}b'$.\\
     If $x=\pm1$ then $n|2x=a+B=\pm 2$, consequently $n\in\{1,2\}$.\\
   \indent If $n=1$ then $C$ is  trivial subgroup.\\
   \indent If $n=2$ then $n=x^{2}-\frac{u_{2}^{2}v_{2}^{2}\Delta}{4d^{2}}b'^{2}=2\Longleftrightarrow -\frac{u_{2}^{2}v_{2}^{2}\Delta}{4d^{2}}b'^{2}=1\Longleftrightarrow \frac{u_{1}^{2}v_{2}^{2}
   + 4v_{1}v_{2}u_{2}^{2}}{4d^{2}}b'^{2}= -1.$  $(*)$\\
   \indent If $2|b'$ then we have the equation $\Big(\frac{b'}{2}\Big)^{2}\cdot\Big\lbrack u_{1}^{2}\Big(\frac{v_{2}}{d}\Big)^{2}
   + v_{1}v_{2}\Big(\frac{u_{2}}{d}\Big)^{2}\Big\rbrack=-1$, hence $\frac{b'}{2}=\pm 1$, i.e. $b'=\pm 2$.\\
   The equation (*) becomes $u_{1}^{2}v_{2}^{2} + 4v_{1}v_{2}u_{2}^{2}=-d^{2}$ and, by using the fact that $u_{2}=du'_{2}$
   and $v_2=d\cdot (v'_{2})$ we obtain that $u_{1}^{2}(v'_{2})^{2} + 4v_{1}v_{2}(u'_{2})^{2}=-1$.
   But a perfect square is congruent to either $0({\rm mod} 4)$ or $1({\rm mod} 4)$,
   hence from the previous equation $-1\equiv 0({\rm mod} 4)$ or $-1\equiv 1({\rm mod} 4)$ respectively, contradiction.\\
   It remains that $2\not|b'$. Since $a=\pm1-\frac{u_{1}v_{2}}{2d}b'\in\mathbb{Z}$ we get $2d|u_{1}v_{2}$. The equation
   $(*)$
   becomes
   $n=b'^{2}\cdot\Big\lbrack\Big(\frac{u_{1}v_{2}}{2d}\Big)^{2}+
   + v_{1}v_{2}\Big(\frac{u_{2}}{d}\Big)^{2}\Big\rbrack=-1$, consequently $b'^{2}=1$.
   In conclusion $n=2$ implies $2d|u_{1}v_{2}$ and $\frac{u_{2}^{2}v_{2}^{2}\Delta}{4d^{2}}=-1$.\\
   This case is realized for $b'=\pm 1$ and $a=\pm1-\frac{u_{1}v_{2}}{2d}b'$.\\
 If $x=\pm \frac{3}{2}$ we have that $n|2x=a+B=\pm 3$, hence $n\in\{1,3\}$.\\
    \indent If $n=1$ then $C$ is  trivial subgroup.\\
    If $n=3$  we have that  $n=x^{2}-\frac{u_{2}^{2}v_{2}^{2}\Delta}{4d^{2}}b'^{2}=\frac{9}{4}-\frac{u_{2}^{2}v_{2}^{2}\Delta}{4d^{2}}b'^{2}|2x=\pm 3$.
    The divisibility occurs if and only if $\frac{u_{2}^{2}v_{2}^{2}\Delta}{4d^{2}}b'^{2}=-\frac{3}{4}\Longleftrightarrow b'^{2}\cdot\frac{u_{1}^{2}v_{2}^{2}
   + 4v_{1}v_{2}u_{2}^{2}}{d^{2}}=-3$. It follows that either $b'^{2}=1$ or $b'^{2}=3$. The case  $b'^{2}=3$ is impossible.
   Consequently, it remains that $b'^{2}=1$ and
    $\frac{u_{2}^{2}v_{2}^{2}\Delta}{4d^{2}}=-3$.\\
   Since $x=\pm \frac{3}{2}$ we obtain that $a=\pm \frac{3}{2}-\frac{u_{1}v_{2}}{2d}b'\in\mathbb{Z}$, consequently $2d\nmid u_{1}v_{2}.$
    This case is realized for $b'=\pm 1$ and $a=\pm\frac{3}{2}-\frac{u_{1}v_{2}}{2d}b'$.\\
   If $x=\pm \frac{1}{2}$ then $n|2x=a+B=\pm 1$, consequently $C$ is trivial subgroup.\\

    Vice-versa, if $2d| u_{1}v_{2}$ or $2d\nmid u_{1}v_{2}$ one can easily check that there
    exists $n$ with the properties specified in the Proposition.

    The fact that such an $n>1$ is unique is also easy to check:\\
    If $2d| u_{1}v_{2}$ and $\frac{-u_{2}^{2}v_{2}^{2}\Delta}{4d^{2}}>1$, there exists $C$ cyclic of order $n=\frac{-u_{2}^{2}v_{2}^{2}\Delta}{4d^{2}}$.\\
    If $2d| u_{1}v_{2}$ and $\frac{-u_{2}^{2}v_{2}^{2}\Delta}{4d^{2}}=1$, there exists $C$ cyclic of order 2.\\
    If $2d\nmid u_{1}v_{2}$ and $n=\frac{-u_{2}^{2}v_{2}^{2}\Delta}{d^{2}}>1$, there exists $C$ cyclic of order $n=\frac{-u_{2}^{2}v_{2}^{2}\Delta}{d^{2}}$.\\
    If $2d\nmid u_{1}v_{2}$ and $n=\frac{-u_{2}^{2}v_{2}^{2}\Delta}{d^{2}}=1$ we have that $\frac{u_{1}^{2}v_{2}^{2}
   + 4v_{1}v_{2}u_{2}^{2}}{d^{2}}=-1$. Since $d|u_{2}$ and $d|v_{2}$ it follows that there exist $u'_{2},v'_{2}\in\mathbb{Z}$ such that
    $u_{2}=du'_{2}$ and $v_{2}=dv'_{2}$. The equality $\frac{u_{1}^{2}v_{2}^{2}
   + 4v_{1}v_{2}u_{2}^{2}}{d^{2}}=-1$ means  $u_{1}^{2}(v'_{2})^{2}
   + 4v_{1}v_{2}(u'_{2})^{2}=-1$, i.e. $(u_{1}(v'_{2}))^{2}\equiv -1({\rm mod} 4)$, contradiction. Consequently, this case is not possible.
\end{proof}

\section{Examples. The cases $n=2$, $n=3$ and $n=5$}

In this section we classify (up to an isomorphism) the elliptic
curves $E$ which admit a subgroup $C\leq (E,+)$ of order either
$2$ or $3$  or $5$ such that
 $\frac{E}{C}\simeq E$.
  We recall that complex elliptic curves are of the form
  $\frac{\mathbb{C}}{L}$ for some
$L=\mathbb{Z}+\mathbb{Z}\tau\subset\mathbb{C}$ where $\tau\in G=\Big\lbrace z=x+iy: -\frac{1}{2}\leq x<\frac{1}{2},|z|\geq 1 \Big\rbrace.$\\
  Let $E$ be an elliptic curve satisfying the condition of Theorem \ref{theorem:1},i).\\
  \indent $E$ is isomorphic to an elliptic curve $E'=\frac{\mathbb{C}}{L}$, where $L=\mathbb{Z}+\mathbb{Z}\tau$ and $\tau\in G$.
  Since an isomorphism $u: E\longrightarrow E'$ is of the type $u(z)=A\cdot z, A\in SL_{2}(\mathbb{Z})$ one easily obtains that
   $E'$ satisfies the condition of Theorem \ref{theorem:1},i).
   Hence we can assume (up to an isomorphism) that $E$ is of the form $\frac{\mathbb{C}}{L}$ with $L=\mathbb{Z}+\mathbb{Z}\tau\subset\mathbb{C}$
  and $\tau\in G$.
  Moreover, we observe that if $\tau=x+iy\in G$, from $|x|\leq\frac{1}{2}$ and $|z|=x^{2}+y^{2}\geq 1$ it follows that $y^{2}\geq \frac{3}{4}$.\\
   Let $\tau^{2}-u\tau-v=0,u,v\in\mathbb{Q}, \Delta=u^{2}+4v<0$ and $\tau\in G$. Then $\tau=\frac{u\pm\ i \sqrt{|\Delta|}}{2}$ and,  since $\tau\in G$,
   we get that  $-1\leq u<1$ and $|\Delta|\geq 3$.\\
   Since $\Delta=u^{2}+4v<0$ we have that $v<0$. Without loss of generality, we can assume that $v_{2}>0,v_{1}<0$ and $u_{2}>0$.
By using  Theorem \ref{theorem:1} and these restrictions imposed
to $\tau$ we obtain the following results:

  \begin{proposition}\label{proposition:1}
  There are exactly $3$ elliptic curves $E$ (up to an isomorphism) which admit at least one subgroup $C\leq(E,+)$ of order $2$ such that $\frac{E}{C}\simeq E$.\\
  If we denote by  $L=\mathbb{Z}+\mathbb{Z}\tau$, they are:\\
  a) $E=\frac{\mathbb{C}}{L},\tau^{2}=-1$;\\
  b) $E=\frac{\mathbb{C}}{L},\tau^{2}=-2$;\\
  c) $E=\frac{\mathbb{C}}{L},\tau^{2}=-\tau-2$.\\
  \end{proposition}

   \begin{remark}
The Fricke involution $w_2$ of $Y_0(2)$ has 2 fixed points  and
they correspond to the cases a) and b); note that $\nu(2)=2$.
  \end{remark}

\begin{proof}
  From Theorem \ref{theorem:1} we have:
  \begin{equation}
   2=aB-bA=\Big(a+\frac{u_{1}v_{2}}{2d}b'\Big)^{2}-\frac{u_{2}^{2}v_{2}^{2}\Delta}{4d^{2}}b'^{2}
  \end{equation}
   Since $\Delta\leq -3$ we obtain $2=\Big(a+\frac{u_{1}v_{2}}{2d}b'\Big)^{2}-\frac{u_{2}^{2}v_{2}^{2}\Delta}{4d^{2}}b'^{2}\geq\Big(a+\frac{u_{1}v_{2}}{2d}b'\Big)^{2}+ \frac{3u_{2}^{2}v_{2}^{2}}{4d^{2}}b'^{2}\geq\frac{3u_{2}^{2}v_{2}^{2}}{4d^{2}}b'^{2}\geq\frac{3b'^{2}}{4}$, hence $b'^{2}\leq\frac{8}{3}$.
   Since $b'\in\mathbb{Z}$ we get that $b'\in\{0,\pm 1\}$.\\
   If $b'=0$ the equation (10) becomes $2=a^{2}$, absurd.\\
   If $b'=\pm 1$, the equation (10) leads to $2\geq\frac{3}{4}\cdot\Big(\frac{u_{2}v_{2}}{d}\Big)^{2}$.\\
    Denote by $\frac{u_{2}v_{2}}{d}=y\in\mathbb{Z}$. From the above inequality we get that $|y|\leq 1$.\\
    \indent I) If $y=0$ then $u_{2}v_{2}=0$, absurd.\\
    \indent II) If $y=\pm 1$ then $\frac{u_{2}v_{2}}{d}=\pm 1$.\\
    Since $u_{2}>0$ and $v_{2}>0$ we have $\frac{u_{2}v_{2}}{d}=1\Longleftrightarrow u_{2}\cdot\frac{v_{2}}{d}=1\Longleftrightarrow v_{2}\cdot\frac{u_{2}}{d}=1$.
    We obtain that $u_{2}=\frac{v_{2}}{d}=1$ and $v_{2}=\frac{u_{2}}{d}=1$, consequently $u_{2}=v_{2}=d=1$.
     The equation (10) becomes:
     \begin{equation}
   2=\Big(a+\frac{u_{1}}{2}b'\Big)^{2}-\frac{\Delta}{4}
  \end{equation}
  Consequently  $\tau^{2}-u_{1}\tau-v_{1}=0,u_{1},v_{1}\in\mathbb{Z}, \Delta=u_{1}^{2}+4v_{1}<0$. From $\tau\in
  G$ and
   $\tau=\frac{u_{1}\pm\ i \sqrt{|\Delta|}}{2}$ we get that $-1\leq u_{1}<1$. Since $u_{1}\in\mathbb{Z}$ we have that $u_{1}\in\{-1,0\}$.
    We distinguish the following cases:\\
 \indent a)$u_{1}=0$.\\
  The equation (12) becomes $2=a^{2}-\frac{\Delta}{4}=a^{2}-\frac{4v_{1}}{4}=a^{2}-v_{1}$. But $v_{1}<0$ and $v_{1}\in\mathbb{Z}$, consequently $a^{2}=-v_{1}=1$ sau $a^{2}=0, -v_{1}=2$.\\
  If $a^{2}=-v_{1}=1$ it follows that $\tau^{2}=-1$ and $(a,b')=(\pm1,\pm1)$ (hence $a$ and $b'$ are co-primes).\\
  If $a^{2}=0, -v_{1}=2$ we obtain that $\tau^{2}=-2$ and $(a,b')=(0,\pm1)$ (hence $a$ and $b'$ are co-primes).\\
 \indent b)$u_{1}=-1$.\\
  The equation (12) becomes $2=\Big(a-\frac{b'}{2}\Big)^{2}-\frac{\Delta}{4}=\Big(a-\frac{b'}{2}\Big)^{2}-\frac{1+4v_{1}}{4}.$\\
   Since $b'=\pm 1$ we distinguish two cases:\\
 \indent \quad b1) If  $b'=1$ we have that $2=\Big(a-\frac{1}{2}\Big)^{2}-\frac{1+4v_{1}}{4}\Longleftrightarrow 8=(2a-1)^{2}-1-4v_{1}\Longleftrightarrow 9=(2a-1)^{2}-4v_{1}$.\\
   On the other hand $v_{1}\leq -1$ hence the only possibility for the previous equation is $v_{1}=-2$ and $2a-1=\pm 1$.
    From $u_{1}=-1$ and $v_{1}=-2$ it follows that $\tau$ satisfies the equation $\tau^{2}+\tau+2=0$.
    Moreover, if $2a-1=1$ it follows that $(a,b')=(1,1)$ ($a$ and $b'$ are co-primes).\\
    If $2a-1=-1$ it follows that $(a,b')=(0,1)$ ($a$ and $b'$ are co-primes).\\
 \indent \quad b2) If  $b'=-1$ we have that $2=\Big(a+\frac{1}{2}\Big)^{2}-\frac{1+4v_{1}}{4}\Longleftrightarrow 8=(2a+1)^{2}-1-4v_{1}\Longleftrightarrow 9=(2a+1)^{2}-4v_{1}$.\\
   Since $v_{1}\leq -1$ we obtain similarly that $v_{1}=-2$ and $2a+1=\pm 1$. Since $u_{1}=-1$ and $v_{1}=-2$ we obtain that $\tau$
    verifies the same equation $\tau^{2}+\tau+2=0$.
    \end{proof}

    \begin{remark} In particular, given a complex elliptic curve $E$ in the form $\bC/\langle 1,\tau \rangle$ such that $\tau\in G$,
     there are  exactly three values of $\tau$ for which $E$ admits an
    endomorphism of degree $2$: $\tau=i;
    \tau=\sqrt{-2};\tau=\frac{-1+\sqrt{-7}}{2}$, (see \cite{Ha},
    Exercise 4.12), emphasizing in this way the validity of our
    results.
    \end{remark}

 \begin{proposition}\label{proposition:2}
  There are exactly $4$ elliptic curves $E$ (up to an  isomorphism) which admit at least one subgroup $C\leq (E,+)$ of order $3$ such that
   $\frac{E}{C}\simeq E$.
  If we put $L=\mathbb{Z}+\mathbb{Z}\tau$, they are:\\
  a) $E=\frac{\mathbb{C}}{L},\tau^{2}=-2$;\\
  b) $E=\frac{\mathbb{C}}{L},\tau^{2}=-3$;\\
  c) $E=\frac{\mathbb{C}}{L},\tau^{2}=-\tau-1$;\\
  d) $E=\frac{\mathbb{C}}{L},\tau^{2}=-\tau-3$.\\
  \end{proposition}

   \begin{remark}
The Fricke involution $w_3$ of $Y_0(3)$ has 2 fixed points  and
they correspond to the cases b) and c);  note  that $\nu(3)=2$.
  \end{remark}

  \begin{proof}
  By using Theorem \ref{theorem:1},i) we have that:
   \begin{equation}
   3=aB-bA=\Big(a+\frac{u_{1}v_{2}}{2d}b'\Big)^{2}-\frac{u_{2}^{2}v_{2}^{2}\Delta}{4d^{2}}b'^{2}
  \end{equation}
   Since $|\Delta|\geq 3$ we obtain that $3\geq\frac{3u_{2}^{2}v_{2}^{2}}{4d^{2}}b'^{2}\geq\frac{3}{4}b'^{2}$. Consequently $b'^{2}\leq 4$.\\

  \indent I) If  $b'^{2}=4$ both sides of the previous inequality are equal since $3=\frac{3}{4}b'^{2}$, hence $\Delta=-3$ and $\frac{u_{2}v_{2}}{d}=1$.
   Since $u_{2}>0$ and $v_{2}>0$ we get that $u_{2}=v_{2}=d=1$.
  Consequently $\tau$ satisfies the equation $\tau^{2}-u_{1}\tau-v_{1}=0,u_{1},v_{1}\in\mathbb{Z},\Delta=u_{1}^{2}+4v_{1}<0$.
  From $\tau\in G$ and $\tau=\frac{u_{1}\pm\ i \sqrt{|\Delta|}}{2}$ we obtain that $-1\leq u_{1}<1$. Since $u_{1}\in\mathbb{Z}$
  we get that $u_{1}\in\{-1,0\}$. By using the equation $\Delta=-3$ we obtain that $u_{1}^{2}+4v_{1}=-3$, consequently $u_{1}$ is odd.
   It follows that $u_{1}=-1$ and
  $v_{1}=-1$, consequently $\tau$ satisfies $\tau^{2}+\tau+1=0$.

  \indent II) If $b'^{2}=1$ by using the equation (14) we get that $3\geq\frac{3}{4}\cdot\Big(\frac{u_{2}v_{2}}{d}\Big)^{2}$.\\
   We denote by  $\frac{u_{2}v_{2}}{d}=y\in\mathbb{Z}$. Since $u_{2}>0$ and $v_{2}>0$ we have that $y>0$.
   From the above inequality we obtain that $y\in\{1,2\}$.\\
\indent a) If $y=1$ we have that $u_{2}=v_{2}=d=1$.\\
    Consequently $\tau$ satisfies the equation $\tau^{2}-u_{1}\tau-v_{1}=0,u_{1},v_{1}\in\mathbb{Z},\Delta=u_{1}^{2}+4v_{1}<0$.\\
    Since $\tau\in G$ and $\tau=\frac{u_{1}\pm\ i \sqrt{|\Delta|}}{2}$ we get that $-1\leq u_{1}<1$. Since $u_{1}\in\mathbb{Z}$ it follows that $u_{1}\in\{-1,0\}$.
     We distinguish two cases:\\
 \indent \quad a1) If $u_{1}=0$ the equation (14) becomes $3=a^{2}-\frac{1}{4}\Delta=a^{2}-v_{1}$.
     On the other hand $v_{1}<0$ and $v_{1}\in\mathbb{Z}$, consequently either $a^{2}=1,v_{1}=-2$ or $a^{2}=0,v_{1}=-3$.\\
     If $a^{2}=1,v_{1}=-2$ we have that $u_{1}=0$ and $v_{1}=-2$. It follows that $\tau$ satisfies the equation $\tau^{2}=-2$.\\
     If $a^{2}=0,v_{1}=-3$ we have that $u_{1}=0$ and $v_{1}=-3$. It follows that $\tau$ satisfies the equation $\tau^{2}=-3$.\\
  \indent \quad a2) If $u_{1}=-1$ the equation (14) becomes
  $3=\Big(a-\frac{b'}{2}\Big)^{2}-\frac{1}{4}\Delta=\Big(a-\frac{b'}{2}\Big)^{2}-\frac{1}{4}\cdot(1+ 4v_{1})\Longleftrightarrow 13=(2a-b')^{2}-4\cdot v_{1}$.
    From $b'^{2}=1$ we get that $2a-b'$ is odd. Since $v_1\in \bZ$ and $v_1\leq -1$ we obtain that $13\geq (2a-b')^{2}+4$
    and consequently $2a-b'=\pm1$ and $v_1=-3$. It follows that $\tau$ satisfies the equation $\tau^{2}+\tau+3=0$ and
    $(a,b')\in \{(0,\pm1), (0,-1), (-1,-1)\}$ (hence $a$ and $b'$ are co-primes). \\
 \indent b) If $y=2$ we have that $\frac{u_{2}v_{2}}{d}=2$ hence $u_{2}\cdot\frac{v_{2}}{d}=v_{2}\cdot\frac{u_{2}}{d}=2$.\\
     Since $u_{2}>0$ and $v_{2}>0$ we get that $(u_{2},v_{2})\in\{(2,1),(1,2),(2,2)\}$. We distinguish $3$ cases:\\
  \indent \quad b1) If $(u_{2},v_{2})=(2,1)$ the equation (14) becomes $3=\Big(a+\frac{u_{1}b'}{2}\Big)^{2}-\Delta$.\\
    Since $\Delta\leq-3$ we get that $a+\frac{u_{1}b'}{2}=0$ and  $\Delta=-3$. Since $a+\frac{u_{1}b'}{2}=0$ and $b'=\pm 1$
    we obtain that
     $u_{1}$ is even. We have that $\tau^{2}-\frac{u_{1}}{2}\tau-v_{1}=0$, $\tau\in G$ and $-1\leq \frac{u_{1}}{2}<1$. Consequently $-2\leq u_{1}<2$ and,
     since $u_{1}$ is even, it follows that $u_{1}\in\{-2,0\}$.\\
     \indent If $u_{1}=-2$, since $\Delta=-3$ we get that $\frac{u_{1}^{2}}{4}+4v_{1}=-3\Longleftrightarrow v_{1}=-1$.
     In consequence
     $\tau$ satisfies the equation $\tau^{2}+\tau+1=0$.\\
    \indent If $u_{1}=0$, from $\Delta=-3$ we get that $\frac{u_{1}^{2}}{4}+4v_{1}=-3\Longleftrightarrow v_{1}=-\frac{3}{4}$, contradiction.\\
 \indent \quad b2) If $(u_{2},v_{2})=(1,2)$ the equation (14)
     becomes
    $3=(a+u_{1}b')^{2}-\Delta$.\\
     Since $\Delta\leq-3$ we get that $a+u_{1}b'=0$ and $\Delta=-3$.
     We have that $\tau^{2}- u_{1}\tau-\frac{v_{1}}{2}=0$, $\tau\in G$ and $-1\leq u_{1}<1$ hence $u_{1}\in\{-1,0\}$.\\
     \indent If $u_{1}=-1$, from $\Delta=-3$ we get that $u_{1}^{2}+2v_{1}=-3\Longleftrightarrow v_{1}=-2$.
     Consequently
     $\tau$ satisfies the equation $\tau^{2}+\tau+1=0$.\\
     \indent If $u_{1}=0$, from $\Delta=-3$ we get that $u_{1}^{2}+2v_{1}=-3\Longleftrightarrow 2v_{1}=-3$, contradiction.\\
 \indent \quad b3) If $(u_{2},v_{2})=(2,2)$ the equation (14)
     becomes
     $3=\Big(a+\frac{u_{1}b'}{2}\Big)^{2}-\Delta$.\\
    Since $\Delta\leq-3$ we obtain $a+\frac{u_{1}b'}{2}=0$ and $\Delta=-3$. Since $a+\frac{u_{1}b'}{2}=0$ and $b'=\pm 1$ we get that $u_{1}$ is even.
    From $\tau\in G$ and $\tau^{2}- \frac{u_{1}}{2}\tau-\frac{v_{1}}{2}=0$ we obtain that $-2\leq u_{1}<2$. Since $u_{1}$ is even we get that $u_{1}\in\{-2,0\}$.\\
    \indent If $u_{1}=-2$, from $\Delta=-3$ it follows that $\frac{u_{1}^{2}}{4}+2v_{1}=-3\Longleftrightarrow v_{1}=-2$. In
    conclusion
    $\tau$ satisfies the equation $\tau^{2}+\tau+1=0$.\\
    \indent If $u_{1}=0$, from $\Delta=-3$ we obtain that $\frac{u_{1}^{2}}{4}+2v_{1}=-3\Longleftrightarrow v_{1}=-\frac{3}{2}$, contradiction.

     \indent III) If  $b'^{2}=0$ the equation (14) leads to $3=a^{2}$, contradiction.
     \end{proof}

We also obtain the following:

\begin{proposition}\label{proposition:3}
  There are exactly $6$ elliptic curves $E$ (up to an  isomorphism) which admit at least one subgroup $C\leq (E,+)$ of order $5$ such that
   $\frac{E}{C}\simeq E$.
  If we put $L=\mathbb{Z}+\mathbb{Z}\tau$, they are:\\
  a) $E=\frac{\mathbb{C}}{L},\tau^{2}=-1$;\\
  b) $E=\frac{\mathbb{C}}{L},\tau^{2}=-4$;\\
  c) $E=\frac{\mathbb{C}}{L},\tau^{2}=-5$;\\
  d) $E=\frac{\mathbb{C}}{L},\tau^{2}=-\tau-3$;\\
  e) $E=\frac{\mathbb{C}}{L},\tau^{2}=-\tau-5$;\\
  f) $E=\frac{\mathbb{C}}{L},\tau^{2}=-\tau-\frac{3}{2}$.\\
  \end{proposition}

  \begin{remark}
The Fricke involution $w_5$ of $Y_0(5)$ has 2 fixed points  and
they correspond to the cases c) and f); note that
$\nu(5)=h(-20)=2$.
  \end{remark}

\begin{proof}
We use the same reasoning as in the Propositions
\ref{proposition:1} and \ref{proposition:2} hence the proof is
left to the reader as an exercise.
\end{proof}







\begin{remark}
Given a prime number $p$, there are exactly $p+1$ complex elliptic
curves $E$ (up to an isomorphism) which admit at least one
subgroup $C\leq (E,+)$ of order $p$ such that $\frac{E}{C}\simeq
E$. Note also that given a prime number $p$, there are $p+1$
unramified coverings of degree $p$ of a (complex) elliptic curve.
\end{remark}

\[\]

\begin{center}\copyright Bogdan Canepa \& Radu Gaba 2011
\end{center}
\end{document}